\documentclass[reqno,11pt]{amsart}
 \usepackage{mathrsfs}
 \usepackage{amsmath,amssymb}

\usepackage[xcolor=red]{changes}
\usepackage{color}
\usepackage{xcolor}

\newtheorem{remark}{Remark}

\usepackage{ulem,ifthen,xcolor,xkeyval,pdfcolmk}

\usepackage{graphicx}

\makeatletter
\@namedef{Changes@AuthorColor}{red}
\colorlet{Changes@Color}{red}
\makeatother

\definecolor{db}{RGB}{23,20,119}
\definecolor{dg}{RGB}{2,101,15}
\definecolor{plum}{rgb}{0.56, 0.27, 0.52} 

\usepackage{amssymb,amsmath,txfonts}
\usepackage[abbrev]{amsrefs}

\usepackage{amsfonts}
\usepackage{amsmath,amsthm,amscd}
\usepackage{graphicx}

\newcommand{\ddbar}{\sqrt{-1}\, \partial\bar{\partial}}

\newcommand{\R}{\mathbb R}

\newcommand{\eps}{\varepsilon}
\renewcommand{\phi}{\varphi}
\newcommand{\pa}{\partial}

\title[]{Remarks on a result of Chen-Cheng}

\author{Zhiqin Lu}
\author{Reza Seyyedali}

\email[Zhiqin Lu]{zlu@uci.edu}
\email[Reza Seyyedali]{rezaseyyedali@gmail.com}

\date{October 12, 2023}

\newtheorem{theorem}{Theorem}[section]
\newtheorem{proposition}{Proposition}[section]
\newtheorem{lemma}{Lemma}[section]
\begin{document}
\maketitle

% ----------------------------------------------------------------
\begin{abstract}
In their seminal work (\cite{CC}, \cite{CC2}), Chen and Cheng proved apriori estimates for the constant scalar curvature metrics on compact K\"ahler manifolds. They also prove $C^{3,\alpha}$-estimate for the potential of the K\"ahler metrics under boundedness assumption on the scalar curvature and the entropy. 
The goal of this paper is to replace the uniform boundedness of the scalar curvature to the $L^p$-boundedness of the scalar curvature. 
\end{abstract}

\section{Introduction} 
%\added[id=zl]{I suggest change here.} \comment{comments here.}

A fundamental theorem in the realm of complex analysis is the Uniformization Theorem. One of the implications of the Uniformization Theorem is that every compact Riemann surface admits a metric with consistent Gaussian curvature. This principle can be extended in numerous ways to manifolds of higher dimensions. Within complex geometry, the aspiration is to discover canonical metrics on a Kähler manifold, those that align with the complex structure and exhibit curvature with specified characteristics. K\"ahler-Einstein metrics, constant scalar curvature K\"ahler metrics, and extremal metrics are prime examples of such metrics.

The   existence of K\"ahler-Einstein metrics on compact complex manifolds was proved  by Yau for manifolds with a trivial canonical class(\cite{Y1}, \cite{Y2}).  In the case of  negative first Chern classes, both Aubin and Yau independently affirmed the existence of K\"ahler-Einstein metrics(\cite{A}, \cite{Y1}, \cite{Y2}). However, the scenario is most challenging for Fano manifolds, where the first Chern class is positive, and there exist known barriers to the realization of K\"ahler-Einstein metrics. As conjectured by Yau, these barriers should all correlate to the stability of the manifolds.

The challenge concerning Fano manifolds was eventually overcome by Chen, Donaldson, Sun(\cite{CDS1}, \cite{CDS2}, \cite{CDS3}) , and Tian(\cite{T3}) a few years back. Regarding cscK metrics, the Yau-Tian-Donaldson Conjecture proposes that the presence of such metrics corresponds to a form of stability. The cscK metrics scenario is notably more intricate than that of K\"ahler-Einstein metrics, primarily because the constant scalar curvature equation is a fourth-order fully nonlinear elliptic PDE, while our understanding of fourth-order nonlinear PDEs is still limited. In contrast, the K\"ahler-Einstein equation is a second-order fully nonlinear elliptic PDE, a field that has been extensively explored over the years.

Progress in the constant scalar curvature equation had been stagnant until the recent breakthrough  of Chen-Cheng, (\cite{CC}, \cite{CC2}),  who established a priori estimates for cscK equations, providing significant insights that the Kähler potential and all its derivatives of a cscK metric can be  controlled in terms of the relative entropy.

Let $M$ be a K\"ahler manifold of dimension $n$ and $\omega$ be its K\"ahler form. For any K\"ahler potential $\phi,$ define $\omega_{\phi}=\omega + \ddbar \phi$. We consider the equations
 \begin{equation}{\label{eq1}}\omega_{\phi}^n=(\omega + \ddbar \phi)^n=e^F \omega ^n,\quad \sup_{M} \phi =0,\quad
\Delta _{\omega_{\phi}} F= -R+{\rm tr}_{\omega_{\phi}}\eta.\end{equation} where $R$ is the scalar curvature of the metric $\omega_{\phi}$, and $\eta$ is a fixed smooth $(1,1)$-form. The prototype of $\eta$ is the Ricci curvature ${\rm Ric} (\omega)$ of $\omega$. \\

In their papers (\cite{CC}, \cite{CC2}), Chen and Cheng   proved the following:

\begin{theorem}[\cite{CC}, \cite{CC2}]
For any $p \geq 1 ,$  there exists a constant $C$ depends on $n, p,  \omega, \eta, \|R\|_{\infty}$ and $\int_{M} e^F  \sqrt{1+F^2} \omega^n$ such that $||F||_{W^{2,p}},||\phi||_{W^{4,p}} \leq C.$ In particular, $F$ and $\phi$ are uniformly bounded in $\mathcal C^{1,\alpha}$ and $\mathcal C^{3,\alpha}$ respectively for any $\alpha \in (0,1).$
\end{theorem}

With some   modifications to the argument in \cite{CC} 
we slightly generalize the proceeding theorem. 
Namely, we replace the uniform bound on the scalar curvature with the $L^p$-bound for some $p>0$.   

Let $\Phi (t)=\sqrt{1+t^2}$. Define $A_F$ and $A_{R,p}$ by
\[
A_{F}^n= \int_{M} e^F \Phi(F)\,\omega ^n,\quad A_{R,p}^n= \int_{M} e^F\Phi(R)^{p}\omega ^n
\]
for $p>0$.  $A_F$ gives an upper bound for the entropy
\[
\int_M Fe^F\omega^n\leq A_F
\]
and $A_{R,p}$ gives an upper bound for the $L^p$-norm of $R$ with respect to $\omega_\phi$
\[
\left(\int_M |R|^p\omega_\phi^n\right)^{1/p}\leq A_{R,p}^{n/p}.
\]

The main results of this paper are the following Theorems.

\begin{theorem}{\label{mainthm1}}
For any $p>n, $ there exists  a constant $C$ depends on $n,p,\omega,$
 $A_F$, and  $A_{R,p}$ such that $||F||_{\infty} \leq C$ and $||\phi||_{\infty} \leq C.$

\end{theorem}

\begin{theorem}{\label{mainthm2}}
Let $n=\dim M$. Then 
there exist $p_{n}>2n$ depends only on    $n$  such that  $||F||_{W^{2,p_{n}}} \leq C$ and $||\phi||_{W^{4,p_{n}}} \leq C$  for a constant $C$ depending on $n,\omega,$ $\eta$,   $A_F$, and  $A_{R,p_n}$.
 
Moreover, for any $p \geq p_{n},$ there exists a constant $C_{p}$ depends on $n,\omega,$ $\eta$,  $A_F$, and  $A_{R,p}$  such that $||F||_{W^{2,p}} \leq C_{p}$ and $||\phi||_{W^{4,p}} \leq C_{p}.$

\end{theorem}

Note that in Theorem~\ref{mainthm2}, $W^{2,p}$ and $W^{4,p}$ are optimal regularity for $\phi$ and $F$ respectively, because 
of ~\eqref{eq1} and the fact that $R$ is $L^p$ for some $p>0$.

Theorem \ref{mainthm2}   gives apriori $\mathcal C^{3,\alpha}$ and $\mathcal C^{1,\alpha}$ estimate for $\phi$ and $F$ respectively for some $\alpha=\alpha(p,n) \in (0,1)$ by Sobolev Embedding Theorem.\\

%\begin{cor}Assume that the Calabi flow $\omega_{t}$ exists on the interval $[0,T)$ and $\int_{M} |R(\omega_t)|^p \omega_t^n$ is uniformly %bounded for some $p >2n$. Then,  we can extend the flow beyond $T$.\end{cor}
%As a corollary of the results in \cite{CH}, it follows from that the Calabi flow %extends as long as the $L^p-$ norm of the scalar curvature with respect to %the evolving metric is bounded for some $p>2n.$
%fact that $\omega_{t}$ is uniformly bounded in $C^{1,\alpha}$ for some $\alpha > 0.$ Then one can apply the short time existence in \cite{CH}.

The paper is organized as follows.
In Section $2$, we prove Theorem \ref{mainthm1}.  Our argument does not use  the Alexandrov maximum principle and the cut-off function as in Chen-Cheng~\cite{CC} and \cite{CC2}. Instead, we use Ko\l odziej's Theorem to prove the boundedness of the auxiliary functions.
We then prove the result using the classical maximum principle.

In Section $3$, we prove that there is an  
$L^p$-estimate of $n+\Delta\phi$.
The $\mathcal C^2$ estimate is obtained in Section $4$ using Moser iteration.
The arguments in Sections $3$ and $4$ are essentially the same as those in \cite{CC2}.

Throughout this paper, we shall use 
$\int_M f$ to denote $\int_M f\omega^n$, where $\omega$ is the background metric of the manifold. We use $\|f\|_p$ to denote the $L^p$-norm of function $f$ with respect to the background metric $\omega$.

\section{Proof of Theorem \ref{mainthm1}}
The section's main goal is to prove a uniform estimate for $\phi$ and $F$. 
This section's constant $C$  depends on $n=\dim M$, $\omega$, and $\eta$, which may differ line by line.

 \begin{lemma}\label{lemma21}
Let $h:M \to \R$ be a positive function and $\phi$ and $ \nu $ be  K\"ahler potentials such that 
$$(\omega +\ddbar \phi)^n = e^F \omega^n,$$
$$(\omega +\ddbar \nu)^n = e^F h^n \omega^n.$$ Then $\Delta_{\phi} \nu \geq n h-{\rm tr}_{\omega_{\phi}}(\omega).$ Here $\omega_{\phi}=\omega +\ddbar \phi$ and $\Delta_{\phi}$ is the Laplacian with respect to the metric $\omega _ {\phi}.$
\end{lemma}

\begin{proof}
 This follows by applying the AM-GM inequality to ${\rm tr}_{\omega_{\phi}}(\omega +\ddbar \nu)$.

\end{proof}

Let  $\alpha=\alpha(M,\omega)$ be the $\alpha$-invariant of $(M,\omega).$ By definition, for any smooth function  $\varphi: M \to \R$ such that $\omega+\ddbar \varphi >0,$ we have $$\int_M e^{-\frac 12\alpha(\varphi-\sup_{M} \varphi)} \omega^n \leq C$$
for some $C>0$ inepedent to $\phi$.

\begin{theorem}\label{newtheorem}
For any $p >n, $ there exists $\delta_{0}=\delta_{0}$ depending on $n,p,\omega, \eta, A_F, \|R\|_p$ such that 
for any $\delta < \delta_{0}$, we have

$$\int_{M}e^{\left( 1+\delta\right) F} \leq C,$$
	where $C=C(n,p,\delta_0,\omega, \eta, A_{F}, ||R||_{p})$.
\end{theorem}
\begin{proof}
	
For a fixed $p>n,$ We define functions
$\psi$ and $\rho$ as  the solutions of the following:
\begin{align}&{\label{eq2}}(\omega + \ddbar \psi)^n=A_{F}^{-n}e^F \Phi(F)\omega ^n= A_{F}^{-n}  \Phi(F)\omega_\phi ^n,\,\,\,  \sup_{M} \psi=0;\\
	&{\label{eq3}}(\omega + \ddbar \rho)^n=A_{R,p}^{-n}e^F \Phi(R)^{p}\omega ^n=A_{R,p}^{-n} \Phi(R)^{p}\omega_\phi ^n ,\,\,\,  \sup_{M} \rho=0.\end{align}
	
	Let $0<\epsilon \leq 1$ and $u=F+\epsilon \psi+\epsilon \rho -\lambda \phi=v-\lambda \phi$, where $v=F+\epsilon \psi+\epsilon\rho$. Let $\delta>0$. Then by Lemma~\ref{lemma21}, we have 
	\begin{align}\label{basic-1}
		\begin{split}
			&e^{-\delta u}\Delta_{\phi}(  e^{\delta u})\geq \delta   \Delta_{\phi} u   \\
			& \geq \delta   ( -R+{\rm tr}_{\omega_{\phi}}\eta  )+ \epsilon \delta   (nA_{F}^{-1}\Phi(F)^{\frac{1}{n}}-{\rm tr}_{\omega_{\phi}}\omega)
			\\&+ \epsilon \delta   (A_{R,p}^{-1}\Phi(R)^{\frac{p}{n}}-{\rm tr}_{\omega_{\phi}}\omega)-n\delta \lambda  +\delta  \lambda  \,{\rm tr}_{\omega_{\phi}}\omega\\
			& \geq \delta  ( -R+\epsilon n A_{F}^{-1}\Phi(F)^{\frac{1}{n}}+\epsilon A_{R,p}^{-1}\Phi(R)^{\frac{p}{n}}-\lambda n).
		\end{split}
	\end{align}
	The last inequality holds since $\epsilon \leq 1$ and $\lambda= |\eta|_{\omega}+2.$
	
	Let $$\delta_{0}= \lambda^{-1}\min(\alpha,1),$$
	 where $\alpha=\alpha(M,[\omega])$ is the $\alpha$-invariant of $M$.
	We choose $0<\delta<\frac 12 \delta_{0}.$  Fixing $\delta$, we choose $\epsilon>0$ small so that 
	\[
	2(1+\delta )\cdot\epsilon<\min(\alpha,1). 
	\]
	Let 
	\[
	\hat\Phi(F)=\epsilon n A_F^{-1}\Phi(F)^{1/n}.
	\]
	Then  $$\epsilon A_{R,p}^{-1}\Phi(R)^{\frac{p}{n}}-R\geq -C(\epsilon,p),$$ since $A_{R,p}$ is bounded and $p>n$. Therefore, ~\eqref{basic-1} implies that 
	\begin{equation}\label{basic-2}
		\Delta_{\phi} e^{\delta u}\geq \delta e^{\delta u}(\hat\Phi(F)-C)
	\end{equation}
	for some constant $C>0$.
	As a result, we have 
	\[
	\int_M e^{\delta u}(\hat\Phi(F)-C)\omega_\phi^n\leq 0.
	\]
	We let
	\begin{align*}
		&E_1=\{x\mid \hat\Phi(F)-C\geq 1\};\\
		& E_2=\{x\mid \hat\Phi(F)-C< 1\}.
	\end{align*}
	On $E_2$, $F$ is bounded, say    $F\leq C$.  Thus  we have
	\[
	\int_{E_1} e^{\delta u+ F}\leq \int_{E_1}e^{\delta u} (\hat\Phi(F)-C)\omega_\phi^n
	\leq -\int_{E_2} e^{\delta u}(\hat\Phi(F)-C) \omega_\phi^n.
	\]
	Since $\hat\Phi(F)$ is  nonnegative, and on $E_2$, we have   $u\leq C-\lambda \phi$, we have
	\[
	\int_{E_1} e^{\delta u+F}\leq   C\int_{E_2} e^{-\delta\lambda\phi}\leq C\int_{M} e^{-\delta\lambda\phi}\leq C,
	\]
	since $\delta\lambda$ is less than half of the $\alpha$-invariant.  By definition of $u$, we have 
	\[
	\int_{E_1} e^{(1+\delta) F+\epsilon \delta (\psi+\rho)}\leq\int_{E_1} e^{\delta u+F}\leq C.
	\]
	Since 
	\[
	\omega+\sqrt{-1}\pa\bar\pa \frac{\psi+\rho}{2}>0,
	\]
	using the H\"older inequality, we have 
	\[
	\begin{split}
		&
		\int_{E_1} e^{(1+\delta/2)F}=\int_{E_1}e^{(1+\delta/2)F+\frac{(1+\delta/2 )}{1+\delta}\epsilon\delta  (\psi+\rho)}\cdot e^{-\frac{ (1+\delta/2)}{1+\delta}\epsilon\delta  (\psi+\rho)}\\
		&\leq
		\left(\int_{E_1}e^{(1+\delta)F+ \epsilon\delta  (\psi+\rho)}\right)^{\frac{1+\delta/2}{1+\delta}}
		\cdot \left(\int_{E_1}e^{-\frac{1+\delta/2}{ \delta/2}\epsilon\delta  (\psi+\rho)}\right)^{\frac{ \delta/2}{1+\delta}}\leq C,
	\end{split}
	\]
	since $\frac{1+\delta/2}{ \delta/2}\epsilon\delta$ is less than hal  f of the $\alpha$-invariant.  
	Combining the above with the fact that $F$ is bounded on $E_2$,   we    have 
	\[
	\int_M e^{(1+\delta/2) F}\leq C.
	\]
	
\end{proof}

 The following proof  of Theorem~\ref{mainthm1} is slightly different from that of Chen-Cheng (\cite{CC}). 
 
\begin{proof} [Proof of Theorem \ref{mainthm1}]
As in ~\eqref{eq2}, ~\eqref{eq3}, we define functions
  $\psi$ and $\rho$ as  the solutions of the following:
\begin{align}&{\label{eq2-2}}(\omega + \ddbar \psi)^n=A_{F}^{-n}e^F \Phi(F)\omega ^n= A_{F}^{-n}  \Phi(F)\omega_\phi ^n,\,\,\,  \sup_{M} \psi=0;\\
&{\label{eq3-3}}(\omega + \ddbar \rho)^n=A_{R,p'}^{-n}e^F \Phi(R)^{p'}\omega ^n=A_{R,p'}^{-n} \Phi(R)^{p'}\omega_\phi ^n ,\,\,\,  \sup_{M} \rho=0,\end{align}
where $p'=(p+1)/2$.

 We shall use the result of Ko\l odziej~\cite{Ko} to prove that the functions $\phi,\psi,\rho$ are uniformly bounded. 
 
 That $\phi$ is bounded directly follows from Theorem \ref{newtheorem} and Ko\l odziej's Theorem. 

Since $x^{1+\delta} e^{-x}\leq C$ for any real number $x>0$,   for $\delta <\delta_{0}/2,$ we have $$\int_M \Phi(F)^{1+\delta}e^{\left( 1+\delta\right)F }\leq C(n,p,\delta_0,\omega, \eta, A_{F}, ||R||_{p}).$$
Hence, Ko\l odziej's Theorem implies that $\psi$ is uniformly bounded.

Finally,  we prove that  $\rho$ is uniformly bounded.
Let $0<\sigma < \delta <\delta_{0}/2$ and $a=1+\sigma.$
We have 

\begin{align}
	\begin{split}
	&\int_M \Phi(R)^{a p' }e^{aF}=\int_M \Phi(R)^{a p' }e^{\sigma F} \omega_\phi ^n
	\\& \leq \left(\int_M \Phi(R)^{ap'\frac{\delta}{\delta-\sigma} } \omega_\phi ^n \right)^{\frac{\delta-\sigma}{\delta}} \left(\int_M e^{\delta F}\omega_\phi^n\right)^\frac{\delta}{\sigma}\\& 
	\leq C  \left(\int_M \Phi(R)^{a p'\frac{\delta}{\delta-\sigma} } \omega_\phi ^n \right)^{\frac{\delta-\sigma}{\delta}}.
		\end{split}
\end{align}
The first inequality follows from H\"older inequality and the last inequality follows from Theorem \ref{newtheorem}.
Now choose $\sigma $ sufficiently small such that 
$ap'\frac{\delta}{\delta-\sigma}< p.$ Therefore, H\"older inequality implies that 

$$\int_M \Phi(R)^{a p' }e^{aF}\leq C\int_M \Phi(R)^{p }\omega_{\phi}^n.$$
This,  together with Ko\l odziej's Theorem implies that $||\rho||_{\infty}\leq C=C(n,\omega, \eta A_{F},A_{R,n+1}).$

Let  $u=F+  \psi+\rho -\lambda \phi$. Then we have
\[
\Delta_\phi u\geq R+\epsilon n A_{F}^{-1}\Phi(F)^{\frac{1}{n}}+\epsilon A_{R}^{-1}\Phi(R)^{\frac{n+1}{n}}-C\geq \epsilon n A_{F}^{-1}\Phi(F)^{\frac{1}{n}}-C.
\]
Let $x_0$ be a maximum point of $u$. Then by the above,
\[
F(x_0)\leq C.
\]
As a result, for any $x\in M$, we have 
\[
u(x)\leq u(x_0)=F(x_0)+\psi(x_0)+\rho(x_0)-\lambda \phi(x_0)\leq C.
\]
This implies that $F(x)\leq C$. 

Now let 
$u'=-F+  \psi+\rho -\lambda \phi$. Then by a similarly computation, we have
\[
\Delta_\phi u'\geq \epsilon n A_{F}^{-1}\Phi(F)^{\frac{1}{n}}-C.
\]
The same argument would imply that $F\geq -C.$  This completes the proof of the theorem.
 \end{proof}

\section{$W^{2,p}$ estimate}
In this section, we prove that for any $p>0$, $n+\Delta\phi$, where $\phi$ is the solution of Equations~\ref{eq1}, is in $L^p(M)$.

 This section's constants $C$   and $C_i$ depend on $n=\dim M$, $p>0$,   $\omega$, and $\eta$, which may differ line by line.

\begin{theorem}\label{thmW2p}
Let $\gamma=\frac{n}{n-1} $ and $p$ be a positive number. %Assume that $\|\phi\|_{L^\infty}$ and $\|F\|_{L^\infty}$ are %bounded. If for some $p$, we have \[
%\|\nabla\phi\|_{L^{2+\frac {4p}{\gamma}}}+\|R\|_{L^(n-1)
%(2p+\gamma)}\leq C,\]
Then
\[
\int_M (n+\Delta\phi)^p\leq C.
\]
 where $C$ depends on $n,p, \omega, \eta, \|\phi\|_{\infty}, \|F\|_{\infty}$ and $\|R\|_{\frac{(n-1)p}{\gamma}}.$
\end{theorem}
To prove Theorem \ref{thmW2p}, we first prove the following gradient estimate.

\begin{proposition}\label{gradientestimate}

For any $p \geq  1$,  there exists a constant $C$ depends on  $n,p, \omega$, $\eta$, $\|\phi\|_{\infty}, \|F\|_{\infty}$, and $\|R\|_{(n-1)p}$ such that
\[
\|\nabla\phi\|_{{2p}} \leq c_1+c_2\|R\|^{(n-1)/2}_{(n-1)p}.
\]

\end{proposition}

\begin{proof}
Let
\[
u=e^{-(F+\lambda\phi)+\frac 12\phi^2}(|\nabla\phi|^2+K),
\]
where $K$ is an absolute constant (for example, we can take $K=10$). Then we have  
\[
\Delta_\phi u\geq C u^{\frac{n}{n-1}} -(c+|R|)u
\]
by ~\cite{CC}*{page 918, equation (2.31)}, where $C,c$ are positive constants depending on $n,p, \omega, \eta, \|\phi\|_{\infty}, \|F\|_{\infty}$. Let $p>0$ and let $\gamma$ be defined in Theorem~\ref{thmW2p}. Then we have
\[
\frac{1}{p+1}\Delta_\phi u^{p+1}=u^p\Delta_\phi u+ p u^{p-1} |\nabla_\phi u|^2
\geq u^p\Delta_\phi u\geq Cu^{p+ \gamma} -(c+|R|) u^{p+1}.
\]
Using   Young's inequality $|R| u^{p+1}\leq |R|^{(p+\gamma)(n-1)}+u^{p+\gamma}$, we have 
\[
\frac{1}{p+1}\Delta_\phi u^{p+1}\geq C u^{p+\gamma}-C_1-C_2 |R|^{(n-1)(p+\gamma)}.
\]
Integrating the above inequality to the volume form $\omega_\phi^n$, we have
\[
C\int_M u^{p+ \gamma}\omega_\phi^n\leq C_1+C_2 \int_M |R|^{(n-1)(p+\gamma)}.
\]
Since $F$ is bounded, $\omega_\phi^n$ and $\omega^n$ are equivalent. 
Thus we have 
\[
\|u\|_{L^{p+\gamma}}\leq c_1+c_2\|R\|_{(n-1)(p+\gamma)}^{n-1}.
\]
The proposition follows by replacing $p+\gamma$ to $p$.
\end{proof}

\begin{proof}[Proof of Theorem \ref{thmW2p}]
Let  
$\alpha$ be a big constant depending on $p$ only, and  to be determined later. Let $\lambda$  be a constant depending on $M$. Let
\[
u=e^{-\alpha(F+\lambda\phi)} (n+\Delta\phi).
\]
By Yau's estimate, we have 
\[
\begin{split}
&
\Delta_\phi u\geq e^{-(\alpha+\frac{1}{n-1})F-\alpha\lambda\phi}\left(\frac{\lambda\alpha}{2}-C\right)(n+\Delta\phi)^{1+\frac{1}{n-1}}\\
&-\lambda\alpha n e^{-\alpha(F+\lambda\phi)}(n+\Delta\phi)+ \alpha  e^{-\alpha(F+\lambda\phi)}R(n+\Delta\phi)
+  e^{-\alpha(F+\lambda\phi)}(\Delta F-R_\omega),
\end{split}
\]
where $R_\omega$ is the scalar curvature of the metric $\omega$. 
%\begin{enumerate}
%\item $\|F\|_{L^\infty}$ and $\|\phi\|_{L^\infty}$ are bounded; %\\[-1.6ex]
%\item $\frac{\lambda\alpha}{2}-C\geq \frac{\lambda\alpha}{4}$;%\\[-1.6ex]
%\item we fix $p\gg 0$ and choose $\lambda\gg p$ and therefore we %can safely drop all constants. 
%\end{enumerate}
By choosing $\lambda$ big enough such that $\frac{\lambda\alpha}{2}-C\geq \frac{\lambda\alpha}{4},$ 
  we have
\[
\Delta_\phi u \geq C_1 u^\gamma-C_2|R|u+e^{-\alpha(F+\lambda\phi)}\Delta F-C_3.
\]
We then have 
\[
\begin{split}
&
\frac{1}{p+1}\Delta_\phi u^{p+1}=u^p\Delta_\phi u+p u^{p-1}|\nabla_\phi u|^2\\
&\geq
p u^{p-1}|\nabla_\phi u|^2+
C_1 u^{p+\gamma}-C_2|R|^{(p+\gamma)(n-1)}+ e^{-\alpha (F+\lambda\phi)}\Delta F \,u^p-C_3,
\end{split}
\]
where we used Young's inequality.
Integrating the above to the volume form $\omega_\phi^n$ and using the fact that $F$ is bounded and $R$ is in $L^{(p+\gamma)(n-1)}$, we have 
\begin{equation}\label{p1}
C_1\int_M u^{p+\gamma} +p \int_M u^{p-1}|\nabla_\phi u|^2 \leq C_3-\int_M e^{-\alpha (F+\lambda\phi)}\Delta F\,u^p.
\end{equation}
Using integration by parts and the fact that both $F,\phi$ are bounded, we have
\begin{equation}\label{p2}
\begin{split}
& -\int_M e^{-\alpha (F+\lambda\phi)}\Delta F u^p \\
&\leq -C-4\alpha \int_Mu^p |\nabla F|^2+C_5\int_M u^{p-1}|\nabla F|(|\nabla u|+ \alpha u|\nabla\phi|),
\end{split}
\end{equation}
where $C_4, C_5$ may depend on $p$ but not $\alpha$. 

By Cauchy-Schwarz inequality, we have
\begin{equation}\label{cau}
|\nabla u|^2=\left(\sum_i \sqrt{1+\phi_{i\bar i}}\cdot\frac{|u_i|}{\sqrt{1+\phi_{i\bar i}}}\right)^2\leq
(n+\Delta\phi)\cdot |\nabla_\phi u|^2.
\end{equation}
Thus $|\nabla u|\leq C u^{1/2}|\nabla_\phi u|$, and hence we have 
\begin{equation}\label{p3}
\frac 12 C_4 \alpha u^p|\nabla F|^2-C_5 u^{p-1}|\nabla F||\nabla u|+\frac 12 p u^{p-1}|\nabla_\phi u|^2
\geq 0
\end{equation}
if we choose $\alpha$ large enough. Similarly, we have 
\begin{equation}\label{p4}
\frac 12 C_4\alpha u^p|\nabla F|^2-C_5 \alpha u^p |\nabla F||\nabla\phi|+C_5\alpha u^p|\nabla\phi|^2\geq 0.
\end{equation}
Using the above two inequalities, we get  
\[
 \int_M u^{p+\gamma}\leq C_3+C_4\int_M u^p|\nabla\phi|^2.
\]
 Using   Young's Inequality, we get 
 \[
 \int_M u^p|\nabla\phi|^2
 \leq
\frac 12  \int_M u^{p+\gamma}+C_5\int_M |\nabla\phi|^{2(p+\gamma)/\gamma}
 \]
By Proposition~\ref{gradientestimate}, we have
\[
\int_M |\nabla\phi|^{2(p+\gamma)/\gamma}\leq \left(c_1+c_2\|R\|^{(n-1)/2}_{(n-1)^2(p+\gamma)/n}\right).
\]
   Since $(n-1)^2(p+\gamma)/n\leq (p+\gamma)(n-1)$, the theorem follows from the monotonicity of the $L^p$-norm.

\end{proof}

%\begin{theorem}
%Assume that $\|\phi\|_{L^\infty}$ and $\|F\|_{L^\infty}$ are %bounded. If for some $p$, we have
%\[
%\|\nabla\phi\|_{L^{2+\frac {4p}{\gamma}}}+\|R\|_{L^(n-1)%(2p+\gamma)}\leq C,
%\]
%then
%\[
%\int (n+\Delta\phi)^p\leq C.
%\]
%\end{theorem}

\section{$\mathcal C^2$-estimate}
In this section, we shall give the $\mathcal C^2$ and high-order estimates.  This section's constants $C$   and $C_i$ depend on $n$,    $\omega$, and $\eta$, which may differ line by line. But contrary to the previous section, these constants are independent of $p>0$.
\begin{theorem}\label{thmC2} 
For each $n$, 
there exist positive numbers $p_{n}$,  $q_{n}$ (depending only on $n$) and $C$ such that  $\|n+\Delta\phi\|_{\infty} \leq C$. Here $C$ depends on $n, \omega, \eta, \|\phi\|_{\infty}, \|F\|_{\infty}, \|R\|_{p_{n}}$, and $\|n+\Delta \phi\|_{q_{n}}.$
\end{theorem}

We start with a Sobolev-type of inequality proved in~\cite{CC}.
\begin{lemma}\label{sobolev}
Let $n$ be the complex dimension of $M$.
Then for any $\epsilon\in (0, \frac{1}{n+1}),$ there exists a constant $C$ depends on $\omega$ and $\epsilon$ such that

$$\|u\|_{\beta}^2\leq C\left(\| n+\Delta\phi\|^2_{\frac{1-\epsilon}{\epsilon}}\int_{M} |\nabla_\phi u|_{\phi}^2 +\|u\|_1^2\right),  $$
where  $\beta=2 \Big(1+\frac{1-(n+1)\epsilon}{n-1+\epsilon}\Big)=\frac{2n(1-\eps)}{n-1+\eps}$.
\end{lemma}
 
\begin{proof} 
The proof is given in \cite{CC}. For the reader's convenience, we include the argument here.  We have  the following Sobolev inequality
\[
\int_M |u|^{2n/(2n-1)}\leq C\left(\int_M|\nabla u|+\int_M |u|\right)^{\frac{2n}{2n-1}}.
\]
Replacing $u$ by $u^{\frac{2n-1}{2n}\beta}$ in the above inequality,  and by interpolation, we get
 \begin{equation}\label{4}
 \int_M |u|^\beta\leq C\left(\int_M |\nabla u|^{2\alpha}+\left(\int_M |u|\right)^{2\alpha }\right)^{\frac{\beta}{2\alpha}},
 \end{equation}
where $\alpha=1-\eps$.

By Cauchy-Schwarz inequality, we have
\[
|\nabla u|^2=\left(\sum_i \sqrt{1+\phi_{i\bar i}}\cdot\frac{|u_i|}{\sqrt{1+\phi_{i\bar i}}}\right)^2\leq
(n+\Delta\phi)\cdot |\nabla_\phi u|^2.
\]
Thus using~\eqref{4} , we have 

\[
\begin{split}
\left(\int_M | u|^\beta\right)^{\frac{2\alpha}{\beta}}
&\leq C\left( \int_{M} |\nabla u|^{2\alpha} +\left(\int_M u\right)^{2\alpha}\right)\\&\leq C\left(\int_M |\nabla_\phi u|^{2\alpha} (n+\Delta\phi)^{\alpha}
+\left(\int_M |u|\right)^{2\alpha}\right)\\
& \leq C \left(\int_M |\nabla_\phi u|^{2}\right)^{\alpha}\left(\int_M (n+\Delta\phi)^{\frac{\alpha}{1-\alpha}} \right)^{1-\alpha}+C\left(\int_M |u|\right)^{2\alpha}.
\end{split}
\]

\end{proof}

\begin{proof}[Proof of Theorem \ref{thmC2}]
We let
\[
u=e^{F/2}|\nabla_\phi F|^2_\phi+(n+\Delta\phi)+1.
\]
Then by~\cite{CC}*{Equation (4.13)}, we have
\begin{equation}\label{4-1}
\Delta_\phi u\geq -C(n+\Delta\phi)^{n-1} u+2 e^{F/2}\langle \nabla_\phi F, \nabla_\phi\Delta_\phi F\rangle-C|R|u -C.
\end{equation}
Multiplying ~\eqref{4-1} by $u^{2p}$ and integrating by parts and using the fact that $F$ is bounded, we have
\begin{equation}\label{4-2}
\begin{split}
&
2p\int_M u^{2p-1}|\nabla_\phi u|^2\omega_\phi^n\leq C\int_M (n+\Delta\phi)^{n-1} u^{2p+1}\\
&+C\int_M|R| u^{2p+1}+C\int_Mu^{2p}
-2\int_M e^{F/2}\langle \nabla_\phi F, \nabla_\phi\Delta_\phi F\rangle u^{2p}\omega_\phi^n.
\end{split}
\end{equation}

In the above last term, we use the same idea as in the proof of Theorem~\ref{thmW2p} to obtain
\[
\begin{split}
&-\int_M e^{F/2}\langle \nabla_\phi F, \nabla_\phi\Delta_\phi F\rangle u^{2p}\omega_\phi^n\\
&=\int_M e^{F/2}(\Delta_\phi F)^2 u^{2p}\omega_\phi^n+\frac 12\int e^{F/2}(\Delta_\phi F)|\nabla_\phi F|^2 u^{2p}\omega_\phi^n\\
&+2p\int_M e^{F/2}(\Delta_\phi F)\langle \nabla_\phi F,  \nabla_\phi u \rangle u^{2p-1}\omega_\phi^n.
\end{split}
\]
Using the Cauchy-Schwarz inequality, for any $\eps_0>0$, we have 
\begin{equation}\label{4-3}
\begin{split}
&
\int_M e^{F/2}(\Delta_\phi F)\langle \nabla_\phi F,  \nabla_\phi u \rangle u^{2p-1}\omega_\phi^n\\
&\leq C\eps_0^{-1}\int_M 
(\Delta_\phi F)^2 u^{2p}\omega_\phi^n+\eps_0\int_M |\langle \nabla_\phi F, \nabla_\phi u\rangle |^2u^{2p-2}\omega_\phi^n
\\&
\leq C\eps_0^{-1}\int_M 
(\Delta_\phi F)^2 u^{2p}+C\eps_0\int_M |\nabla_\phi u|^2 u^{2p-1}.
\end{split}
\end{equation}As a result, we have 
\[
\begin{split}
&
-\int_M e^{F/2}\langle \nabla_\phi F, \nabla_\phi\Delta_\phi F\rangle u^{2p}\omega_\phi^n
\leq C\eps_0\int_M |\nabla_\phi u|^2 u^{2p-1}\\
&+ C(\eps_0^{-1}+1)\int_M 
(\Delta_\phi F)^2 u^{2p}+C_2p\int_M|\Delta_\phi F|  u^{2p+1}.
\end{split}
\]
By choosing $\eps_0$ small enough, from ~\eqref{4-2}, we have 
\begin{equation}\label{4-4}
\begin{split}
&
p\int_M |\nabla_\phi u|^2u^{2p-1}\leq C_1\int_M (n+\Delta\phi)^{n-1} u^{2p+1}\\
&+C_2\int_M|R| u^{2p+1}+C_3\int_M  (\Delta_\phi F)^2 u^{2p}+C_4p\int_M|\Delta_\phi F|  u^{2p+1}.
\end{split}
\end{equation}
Using the equation~\eqref{eq1}, 
we have
\[
|\Delta_\phi F|\leq |R|+|{\rm Tr}_{\omega_\phi}\eta|\leq |R|+C(n+\Delta\phi)^{n-1}.
\]
Therefore from ~\eqref{4-4}, we obtain
 \[
\int_M |\nabla_\phi u|^2 u^{2p-1}\leq C_1\int_M (n+\Delta\phi)^{2n-2} u^{2p+1}+C_2\int_M(1+|R|^2)u^{2p+1}.
\]

Hence,
\[
p^{-2}\int_M|\nabla_\phi u^{p+\frac 12}|^2\leq C \int_M u^{2p-1}|\nabla_\phi u|^2\leq C\int_M ((n+\Delta\phi)^{2n-2}+1+|R|^2) u^{2p+1}. 
\]

Now 
we fix an  $\epsilon \in (0, \frac{1}{n+1}).$ Let $\beta=2(1+\delta)$, where 
\[
\delta=\frac{1-(n+1)\epsilon}{n-1+\epsilon}
\]
 as in Lemma \ref{sobolev}.
Then we have 
\[
\begin{split}
\|u^{p+\frac{1}{2}}\|_{\beta}^2&\leq C  \| n+\Delta\phi\|_{\frac{1-\epsilon}{\epsilon}}\int_{M} |\nabla_\phi u^{p+\frac{1}{2}}|_{\phi}^2
+C\|u^{p+\frac 12}\|_1^2
\\\leq&Cp^2\| n+\Delta\phi\|_{\frac{1-\epsilon}{\epsilon}} \left(  \int_M( (n+\Delta\phi)^{2n-2}+1+|R|^2)u^{2p+1}\right).
\end{split}
\]

On the other hand, let $2<\theta<\beta$ and let $\theta^*=(1-2\theta^{-1})^{-1}$. Then for any function $H$, by the H\"older Inequality, we have
\[
\int_M H u^{2p+1}\leq \|H\|_{{\theta^*}}\cdot \left(\int_M u^{(2p+1)\frac{\theta}{2}}\right)^{\frac{2}{\theta}}.
\]

In particular, we have

\[
\int_M |R|^2 u^{2p+1}\leq \|R\|_{{2\theta^*}}^2\cdot \left(\int_M u^{(2p+1)\frac{\theta}{2}}\right)^{\frac{2}{\theta}}
\]
and 
\[
\int_M (n+\Delta\phi)^{2n-2} u^{2p+1}\leq \|n+\Delta\phi\|_{(2n-2) \theta^*}^{2n-2}\cdot \left(\int_M u^{(2p+1)\frac{\theta}{2}}\right)^{\frac{2}{\theta}}.
\]
Assuming $\|R\|_{2\theta^*}\leq C$, $\|n+\Delta\phi\|_{(2n-2)\theta^*}+ \|n+\Delta\phi\|_{\frac{1-\epsilon}{\epsilon}}\leq C$, we have
$$\|u^{p+\frac{1}{2}}\|_{\beta}^2\leq C p^2 \|u^{p+\frac{1}{2}}\|_{\theta}^2.$$

This implies that for any $p \geq \frac{1}{2}$, we have 
$$\|u\|_{(p+\frac{1}{2})\beta} \leq (Cp^2)^{\frac{2}{2p+1}}\|u\|_{(p+\frac{1}{2})\theta}.$$
Applying Moser's iteration, one obtains 
$$\|u\|_{\infty} \leq C \|u\|_{\theta}.$$

On the other hand $$\|u\|_{\infty}^{\theta}\leq C\|u\|_{\theta}^{\theta}=\int_M |u|^{\theta} \leq C\|u\|_{\infty}^{\theta-1} \|u\|_{1}. $$
which implies that 
 $$\|u\|_{\infty}^{} \leq C \|u\|_{1}\leq C \int_M (|\nabla_\phi F|^2_\phi+(n+\Delta\phi)+1). $$

Since 
 $$\int_M (n+\Delta\phi)\omega_{\phi}^n=n$$ and 
$$\int_M |\nabla_\phi F|^2_{\phi}\omega_{\phi}^n=-\int_M F \Delta_{\phi}F\leq \int_M |F|(|R|+C(n+\Delta \phi)^{n-1})\leq C,$$
we have
$$\|u\|_{\infty} \leq C.$$

\end{proof}

\begin{remark}
Choosing $\epsilon=\frac{1}{2n+1},$ we get  $q_{n}=4n^2-4.$ On the other hand, Theorem \ref{thmW2p} implies that a bound on $\|R\|_{\frac{(n-1)^2(4n^2-4)}{n}}$ gives a bound on 
$\|n+\Delta \phi\|_{4n^2-4}.$ Therefore, we can show that $C$ in the statement of the Theorem \ref{thmC2} depends on $n, \omega, \|\phi\|_{\infty}, \|F\|_{\infty}, \|R\|_{p_{n}},$ where $p_{n}=\frac{4(n-1)^3(n+1)}{n}.$

One might hope to improve the estimate by lowering $p_{n}.$ However, we have not been able to improve the bound yet.

\end{remark}

Now, the proof of Theorem \ref{mainthm2} is straightforward.

\begin{proof}[Proof of Theorem \ref{mainthm2}]

 Suppose that $\phi$ satisfies the equation \eqref{eq1}.  Then Theorem \ref{mainthm1}, Theorem \ref{thmW2p} and Theorem \ref{thmC2} imply that there exists $p_{n}$ such that $$\|n+\Delta \phi\|_{\infty}\leq C=C(n, \omega, \eta, \|R\|_{p_{n}}).$$
This implies that eigenvalues of $\omega_{\phi}=\omega+\ddbar\phi$ are bounded from above by $C$.
On the other hand, by Theorem \ref{mainthm1}, $\|F\|_{\infty}\leq C.$
Therefore, eigenvalues of $\omega_{\phi}=\omega+\ddbar\phi$ are bounded below by a positive constant that only depends on $n, \omega, \eta, \|R\|_{p_{n}}.$ Hence the equation 

$$\Delta _{\omega_{\phi}} F= -R+{\rm tr}_{\omega_{\phi}}\eta$$
is uniformly elliptic. 
  Therefore, DeGiorgi-Nash-Moser Theorem implies that there exists $\alpha \in (0,1)$ such that $\|F\|_{C^\alpha} \leq C. $ This together with the $\mathcal C^2$ bound on $\phi$, we obtain that $\phi$ is bounded in $\mathcal C^{2,\alpha}$(\cite{W}).

 Hence, Carlderon-Zygmond estimate implies that $F$ is bounded in $W^{2,p_{n}}.$ Now differentiating the Monge-Ampere equation implies that $\phi$ is bounded in $W^{4,p_{n}}.$

\end{proof}


\begin{bibdiv}


\begin{biblist}

\bib{A}{article}{
   author={Aubin, T.},
   title={\'Equations du type Monge-Amp\`ere sur les vari\'et\'es
   k\"ahl\'eriennes compactes},
   language={French, with English summary},
   journal={Bull. Sci. Math. (2)},
   volume={102},
   date={1978},
   number={1},
   pages={63--95},
   }









\bib{CC}{article}{
   author={Chen, Xiuxiong},
author={Cheng, Jingrui }
   title={On the constant scalar curvature Kähler metrics I–Apriori estimates },
   doi={ https://doi.org/10.1090/jams/967}
}
 
\bib{CC2}{article}{
   author={Chen, Xiuxiong},
author={Cheng, Jingrui }
   title={On the constant scalar curvature Kähler metrics II–Existence results },
   doi={ https://doi.org/10.1090/jams/966 }
}


\bib{CDS1}{article}{
   author={Chen, X. X.},
   author={Donaldson, S. K.},
   author={Sun, S.},
   title={K\"ahler-Einstein metrics on Fano manifolds. I: Approximation of
   metrics with cone singularities},
   journal={J. Amer. Math. Soc.},
   volume={28},
   date={2015},
   number={1},
   pages={183--197},}

\bib{CDS2}{article}{
   author={Chen, X. X.},
   author={Donaldson, S. K.},
   author={Sun, S.},
   title={K\"ahler-Einstein metrics on Fano manifolds. II: Limits with cone
   angle less than $2\pi$},
   journal={J. Amer. Math. Soc.},
   volume={28},
   date={2015},
   number={1},
   pages={199--234},}


\bib{CDS3}{article}{
   author={Chen, X. X.},
   author={Donaldson, S. K.},
   author={Sun, S.},
   title={K\"ahler-Einstein metrics on Fano manifolds. III: Limits as cone
   angle approaches $2\pi$ and completion of the main proof},
   journal={J. Amer. Math. Soc.},
   volume={28},
   date={2015},
   number={1},
   pages={235--278},}













\bib{Ko}{article}{
   author={Ko\l odziej, S\l awomir},
   title={The complex Monge-Amp\`ere equation},
   journal={Acta Math.},
   volume={180},
   date={1998},
   number={1},
   pages={69--117},
   issn={0001-5962},
   review={\MR{1618325}},
   doi={10.1007/BF02392879},
}
		





\bib{T1}{article}{
   author={Tian, G.},
   title={On Calabi's conjecture for complex surfaces with positive first
   Chern class},
   journal={Invent. Math.},
   volume={101},
   date={1990},
   number={1},
   pages={101--172},
   
}

\bib{T2}{article}{
   author={Tian, G.},
   title={K\"ahler-Einstein metrics with positive scalar curvature},
   journal={Invent. Math.},
   volume={130},
   date={1997},
   number={1},
   pages={1--37},
   
}

\bib{T3}{article}{
   author={Tian, G.},
   title={K-stability and K\"{a}hler-Einstein metrics},
   journal={Comm. Pure Appl. Math.},
   volume={68},
   date={2015},
   number={7},
   pages={1085--1156},
 
}
		
\bib{W}{article}{
   author={Wang, Y. },
   title={On the $C^{2,\alpha}$- regularity of the complex Monge-Amp\`ere equation},
   journal={Math. Res. Lett.},
   volume={19},
   date={2012},
   number={4},
   pages={ 939--946},
 
}



\bib{Y1}{article}{
   author={Yau, S. T.},
   title={Calabi's conjecture and some new results in algebraic geometry},
   journal={Proc. Nat. Acad. Sci. U.S.A.},
   volume={74},
   date={1977},
   number={5},
   pages={1798--1799},
   
}

\bib{Y2}{article}{
   author={Yau, S. T.},
   title={On the Ricci curvature of a compact K\"ahler manifold and the
   complex Monge-Amp\`ere equation. I},
   journal={Comm. Pure Appl. Math.},
   volume={31},
   date={1978},
   number={3},
   pages={339--411},
   }

\bib{Y3}{article}{
   author={Yau, S. T.},
   title={Open problems in geometry},   conference={      title={Differential geometry: partial differential equations on  manifolds (Los Angeles, CA, 1990)},   },
   book={      series={Proc. Sympos. Pure Math.},      volume={54},      publisher={Amer. Math. Soc.},      place={Providence, RI},   },
   date={1993},
   pages={1--28},
   }
		




\end{biblist}
\end{bibdiv}
\end{document}